\documentclass[a4paper,12pt]{amsart}

\usepackage{amsmath,amssymb, amsfonts}
\usepackage{epic}
\usepackage{pgfcore}
\usepackage{graphicx}

\allowdisplaybreaks 

\newtheorem{thm}{Theorem}[section]

\newtheorem{lem}[thm]{Lemma}

\newtheorem{rem}[thm]{Remark}

\author{Hafedh Bousbih}
\address{Universit\'e de Sousse, Institut Sup\'erieur des Sciences Appliqu\'ees et de Technologie de Sousse, Cit\'e Ibn Khaldoun,  Sousse 4003, Tunisie.}\email{hafedh.bousbih@gmail.com}

\author{Mohamed Majdoub}
\address{ Universit\'e de Tunis El Manar, Facult\'e des Sciences de Tunis, LR03ES04 \'Equations aux d\'eriv\'ees partielles et applications, 2092 Tunis, Tunisie.} \email{mohamed.majdoub@fst.rnu.tn}
\thanks{M. M. is grateful to the Laboratory of
PDE and Applications at the Faculty of Sciences of Tunis.}

\title[Shear Thickening Fluids of Second Grade]{Existence and Uniqueness of Strong Solution for Shear Thickening Fluids of Second Grade}
\date{\today}

\begin{document}


\begin{abstract}
In this paper we study the equations governing the unsteady motion of an incompressible homogeneous  generalized second grade fluid subject to periodic boundary conditions.
We establish the existence of global-in-time strong  solutions for shear thickening flows in the two and three dimensional case. We also prove uniqueness of such solution  without any smallness condition on the initial data or restriction on the material moduli.
\end{abstract}




\maketitle

\tableofcontents

\section{Introduction}
The theory and the applications of non-Newtonian fluids have attracted the attention of many scientists for a long time since they  are more appropriate than
Newtonian fluids in many areas of engineering sciences such as geophysics, glaciology, colloid mechanics, polymer mechanics, blood and food rheology, etc.
Fluids with shear dependent viscosity, which can exhibit shear thinning and shear thickening and include the power-law
fluids as a special case, constitute a large class of non-Newtonian fluids. For instance, see \cite{MNR}, \cite{malekraja} and \cite{malekraja07} for  more detailed discussions. \\
There are many models of non-Newtonian fluids which have recently become to be of great interest. Among these models one can cite fluids of differential type. The second grade fluids, which are a subclass of them, have been successfully investigated  in various kinds of flows of  materials such as oils, greases, blood, polymers, suspensions,
and liquid crystals. In the classical incompressible fluids of second grade, it is customary to assume that the Cauchy stress tensor $\mathbf{T}$ is related to the velocity gradient $\nabla u$ and its symmetric part $\mathbf{D} u$ through the relation
\begin{eqnarray}\label{classic cauchy tensor}
\mathbf{T}=-p \mathbf{I} +  \mu \mathbf{A}_{1} + \alpha_{1} \mathbf{A}_{2} +\alpha_{2} \mathbf{A}_{1}^{2},
\end{eqnarray}
where $p$ is the indeterminate part of the stress due to the constraint of incompressibility, $\mu$ is the kinematic  viscosity, and $\alpha_{1}$  and
$\alpha_{2}$ are material moduli which are usually referred to as the normal stress coefficients. The kinematical tensor $\mathbf{A}_{1}$ and $\mathbf{A}_{2}$
are the first and the second Rivlin-Ericksen tensor, respectively, and they are defined through
\begin{eqnarray}
\label{a1}\mathbf{A}_{1}&=&  \nabla u +\nabla u^{t}= 2 \: \mathbf{D} u  = \nabla u + (\nabla u)^{t} ,\\
\label{a2}\mathbf{A}_{2}&=& \frac{d}{dt}\mathbf{A}_{1} + \mathbf{A}_{1} \nabla u + \nabla u^{t} \mathbf{A}_{1},
\end{eqnarray}
 where $ u =( u_{1}, ..., u_{d})$ and  $\nabla u :=(\partial_{j} u_{i})_{1\leq i, j \leq d}$  ($d=2$ or $3$ is the space dimension) denote the velocity vector field and its gradient. Here $\partial_{j} u_{i}$ stands for the partial space derivative of $u_{i}$ with respect to $x_{j}$, $i, j=1,...,d$. The material derivative is given by
\begin{eqnarray}\label{mater deriv}
\frac{d}{dt}(.)= \partial_{t} (.) + (u. \nabla) (.),
\end{eqnarray}
where $\partial_{t}$ is the partial derivative with respect to time and $(u. \nabla)$ the differential operator with respect to the spatial variable defined by $ u. \nabla= \sum_{i=1}^{d} u_{i} \partial_{i}$.\\
Indeed, in \cite{dunn}, Dunn   studied the thermodynamics and stability of second-grade fluids with viscosity $\mu$ depending on the shear rate $|\mathbf{D} u|$ (i.e. the Euclidean norm of the symmetric part of the velocity gradient defined by   $|\mathbf{D}u|= \sqrt{Tr ( \mathbf{D}u)^{2}} $)  and showed that if the fluid is to be compatible with thermodynamics in the sense that all motions of the fluid meet the
Clausius-Duhem inequality and the assumption that the specific Helmholtz free energy of the fluid be minimum in
equilibrium, then $\mu,\; \alpha_{1}$ and $\alpha_{2}$ in (\ref{classic cauchy tensor}) must verify
\begin{eqnarray}\label{firstduncond}
\sqrt{\dfrac{3}{2}} \frac{\mu (|\mathbf{D} u |)}{|\mathbf{D} u | } \leq  \alpha_{1} + \alpha_{2} \leq \sqrt{\dfrac{3}{2}} \frac{\mu (|\mathbf{D} u |)}{|\mathbf{D} u | }.
 \end{eqnarray}
In the  1980s, Man, Kjartanson and coworkers \cite{msksun} and \cite{KSDM} showed  that polycrystalline ice in creeping
flow under pressuremeter tests can be modeled as an incompressible second-grade fluid with the viscosity depending on the shear rate. The constitutive
equation proposed by Man and coworkers, which leads to well-posed initial-boundary-value problems in nonsteady channel
flow \cite{man} and can also model the flow of ice in triaxial creep tests \cite{mansun}, is
\begin{eqnarray}\label{general cauchy tensor}
\mathbf{T}= -p \mathbf{I} + \mu (|\mathbf{D} u |) \mathbf{A}_{1} + \alpha_{1} \mathbf{A}_{2} +\alpha_{2} \mathbf{A}_{1}^{2},
\end{eqnarray}
where $\mu (|\mathbf{D} u |) = \nu |\mathbf{D} u |^{m}$, with $-1 < m < \infty$ and $\nu$ is a material constant.\\
Man and Massoudi \cite{manmas} conducted thermodynamic studies on some classes of generalized second-grade  fluids, which include the class defined by (\ref{general cauchy tensor}). For the case $m > 1$ of this class, they showed that the viscosity function $\mu$ and the
normal-stress coefficients  must verify the inequality
\begin{eqnarray}\label{duncond}
 \mu \geq 0,\; \;\;\;\; \alpha_{1} + \alpha_{2} = 0, \;\;\;\;\;and \;\; \alpha_{1}\geq  0
 \end{eqnarray}
as restrictions imposed by thermodynamics and the assumption that the specific Helmholtz free energy of the fluid be minimum in equilibrium.\\
In this paper we study a class of generalized second-grade fluids with constitutive equation given by (\ref{general cauchy tensor}), which has its normal-stress coefficients satisfy the equation $\alpha_1 +\alpha_2 = 0$ and has a viscosity function $\mu (|\mathbf{D} u |)$ different than
that proposed by Man and coworkers. The mathematical assumptions on the viscosity function that we adopt will be given in Section 3. For the class of
generalized second-grade fluid in question, the thermodynamic restrictions $(\ref{duncond})_1$ and $(\ref{duncond})_3$ are valid \cite{dunn}.\\
The interested reader can find more about the literature of the generalized fluids of second grade and their applications,
for example in \cite{manmas}, \cite{massoudivaidya2008} and \cite{Raja84}. Henceforth we will adopt (\ref{duncond})$_{2}$ and (\ref{duncond})$_{3}$ as assumptions.
Based on these assumptions and the relations (\ref{a1})-(\ref{a2}),
one can deduce that the stress tensor (\ref{general cauchy tensor}) for generalized fluids of second grade could be written in the following form
\begin{eqnarray}\label{general cauchy tensor1}
\mathbf{T}=-p \mathbf{I} &+&  2 \mu (|\mathbf{D} u|) \mathbf{D} u  \nonumber\\
&+& \alpha_{1} [ \partial_{t} (2 \mathbf{D} u) + 2 ( u. \nabla) \mathbf{D} u  + (\nabla u)^{t} \nabla u - \nabla u (\nabla u)^{t}].\;
\end{eqnarray}
Let us mention that by  ``generalized second grade fluid" we mean a fluid of second grade of type (\ref{general cauchy tensor}) (i.e. whose viscosity is a non linear function of the shear rate)  as it will be  considered in this study, and by word `` classical second grade fluid" a second grade fluid (\ref{classic cauchy tensor}) (i.e. whose viscosity is constant).\\
This work is devoted to the mathematical analysis of the  equations governing the flow of a  homogeneous incompressible fluid
 which occupies all space $\mathbb{R}^d (d = 2; 3)$, the Cauchy stress of which is given by formula (\ref{general cauchy tensor1}),
and the velocity of which is $2 \pi$-periodic. We will also assume that the kinematic viscosity $\mu(|\mathbf{D}u|)$ is a nonlinear function of $|\mathbf{D}u|$. The stress tensor $ 2 \mu(|\mathbf{D}u|) \mathbf{D} u $ will be denoted by $\mathbf{S}(\mathbf{D}u)$. Next, we derive the equations modeling our studied system.\\
Let $T$ be an arbitrary positive real number. We set  $[0,T]$ as the time interval and $\Omega : = (0, 2\pi)^{d}$  the  periodic box of dimension $d$. Since the fluid body is taken as incompressible and homogeneous in material, the density of the fluid can be put equal to 1. Then the  momentum equation is expressed as
\begin{eqnarray}\label{momentum}
\frac{du}{dt}= div \, \mathbf{T} + f,
\end{eqnarray}
with the condition of incompressibility
\begin{eqnarray}\label{div oh}
div \, u =0.
\end{eqnarray}
Therefore, with the help of  relation (\ref{mater deriv}), inserting  the expression  (\ref{general cauchy tensor1}) and expanding,    equation (\ref{momentum}) becomes
\begin{eqnarray} \label{equa mere}
\displaystyle
\partial_{t} v + (u \, .\,\nabla) v  - div ( \mathbf{S}(\mathbf{D} u))
+ \sum_{j=1}^{d} v_{j} \nabla u_{j} = -\nabla p + f,\;\;\;\;  v = u  -\alpha_{1}\Delta u.
\end{eqnarray}
One can look at \cite{busuioc1} and more precisely at \cite{busuioc2} to verify that the computations in (\ref{equa mere}), holds true. See also \cite{BLFNL} and \cite{iftimie} where (\ref{equa mere}) is the adopted form for their analysis. These equations come with the  initial data  prescription
\begin{eqnarray}\label{initial cond}
u (0,.)=  u_{0}\in \mathbb{R}^{d},\;\; \mbox{and}\; \;\;div\, u_{0} = 0,
\end{eqnarray}
where $u_{0}$ denotes the initial velocity vector field. We also require that the velocity and the pressure satisfy
\begin{eqnarray}\label{per bound}
\begin{array}{rl}
u_{i}, p :[0,T]\times \mathbb{R}^{d}\longrightarrow  \mathbb{R}\:&\;\;\mbox{are}\;\;\: 2 \pi-\mbox{periodic with respect to}\;\;\; x_{i}, \\
\displaystyle\int_{\Omega} u_{i}\: dx = 0, &\;\displaystyle \int_{\Omega} p \: dx =0, \;  i=1,...,d.
\end{array}
\end{eqnarray}
Recall that the  zero mean value request is  to be assumed for the Poincar\'{e}'s inequality. Note that we are looking for $u:[0,T]\times \mathbb{R}^{d}\longrightarrow \mathbb{R}^{d}$ and $p :[0,T]\times \mathbb{R}^{d}\longrightarrow \mathbb{R}$ that solve the system of equations made up of (\ref{div oh}), (\ref{equa mere}) (\ref{initial cond}) and (\ref{per bound}). For an explanation of the  phenomena of unsteadiness of the considered flow, see  \cite{massoudiphuoc2009}.\\
Before proceeding further, let us clarify some expressions involved in the momentum equation (\ref{equa mere}). Recall that $u$ is  divergence free. Then, in the distributional sense, the $i$ component of the convective term $(u \, .\,\nabla) v$ is defined by
\begin{eqnarray}\label{exp 1}
((u \, .\,\nabla) v)_{i} &=& \sum_{j=1}^{d} u_{j} \partial_{j} v_{i} = \sum_{j=1}^{d} u_{j} \partial_{j} u_{i} - \alpha_{1} \sum_{j,k=1}^{d} u_{j} \partial_{j} \partial_{k k}^{2} u_{i} \nonumber \\
&= &\sum_{j=1}^{d} u_{j} \partial_{j} u_{i} - \alpha_{1} \sum_{j, k =1}^{d}\partial_{j} (u_{j} \partial_{k k}^{2} u_{i})\nonumber\\
&= &\sum_{j=1}^{d} u_{j} \partial_{j} u_{i} - \alpha_{1} \sum_{j, k =1}^{d}\partial_{j k}^{2} (u_{j} \partial_{k} u_{i})
+  \alpha_{1} \sum_{j, k =1}^{d}\partial_{j} ( \partial_{k} u_{j} \partial_{k} u_{i}).\;\;\;
\end{eqnarray}
Analogously the $i$ component (i.e. the partial derivative) of $\sum_{j=1}^{d} v_{j} \nabla u_{j}$ can explicitly be  written as
\begin{eqnarray}\label{exp 2}
(\sum_{j=1}^{d} v_{j} \nabla u_{j})_{i} &=& \sum_{j=1}^{d} u_{j} \partial_{i} u_{j}  - \alpha_{1} \sum_{j, k =1}^{d}  \partial_{k k}^{2} u_{j} \partial_{i} u_{j}\nonumber\\
 &= &\sum_{j=1}^{d} \frac{1}{2} \partial_{i} |u_{j}|^{2} - \alpha_{1} \sum_{j, k =1}^{d} \partial_{k}( \partial_{k}  u_{j} \partial_{i}  u_{j})
 +  \alpha_{1} \sum_{j, k =1}^{d}\partial_{k} u_{j} \partial_{i k}^{2} u_{j}\nonumber\\
&= &\sum_{j=1}^{d}  \frac{1}{2} \partial_{i} (|u_{j}|^{2} + \alpha_{1}|\nabla u_{j}|^{2}) - \alpha_{1} \sum_{j, k =1}^{d}\partial_{k}( \partial_{k}  u_{j} \partial_{i}  u_{j}).\;\;\;
\end{eqnarray}
Therefore (\ref{equa mere}) can be differently expressed  in the following form
\begin{eqnarray} \label{equa mere 2}
\displaystyle
\partial_{t} v + (u \, .\,\nabla) u  - div ( \mathbf{S}(\mathbf{D} u))
- \alpha_{1} \sum_{j,k=1}^{d} \partial_{j k } ( u_{j} \partial_{k} u) +  \alpha_{1} \sum_{j,k=1}^{d} \partial_{j} ( \partial_{k} u_{j} \partial_{k} u)& &\nonumber \\
- \alpha_{1} \sum_{j,k=1}^{d} \partial_{k} ( \partial_{k} u_{j} \nabla u_{j} u)= -\nabla ( p  +  \frac{1}{2}  |u|^{2} + \frac{\alpha_{1}}{2} |\nabla u|^{2})) +  f.& &\;\;\;
\end{eqnarray}
The rest of the paper is planed as follows. In section 2, we introduce  notations and functions spaces. We also outline some technical results that will be used later. In section 3, we set up the assumptions imposed on the stress tensor $\mathbf{S}$. Moreover, we present the resulted properties that will allow us to prove our main result. We also give  a brief explanation of the studied problem
and the considered type of viscosities. Later, we  state the main theorem and relate it with some previous results that are concerned with  the classical second grade fluids and those with shear-rate  dependent viscosities. In section 4, we give the proof of the existence part of our result. Finally, in section 5 we show the uniqueness of the constructed solution.\\
The main feature of the proof  is based  on a  construction of a sequence of approximated
solutions using a discretization in  space of Galerkin's type and  a limit process
using compactness arguments in order to control the nonlinear terms.  In addition, we will also make sense of the weak formulation introduced in the main theorem and conclude by showing that such a solution is  unique via Gronwall's lemma.
\section{Notations and auxiliary results}
Since we deal with a spatial periodic problem on $\mathbb{R}^{d}$, then functions defined on $[0,T]\times \mathbb{R}^{d}$
can be considered as defined on  $[0,T]\times \Omega $. So, we introduce the following spaces. The space $\mathcal{C}^{\infty}_{per}(\Omega)$ consists of
all smooth $2\pi$-periodic functions at each direction $ x_{i}, i=1,...,d$. For $k\in \mathbb{N}$ and  $1\leq q<\infty$, the Lebesgue space $L^{q}_{per}(\Omega)$
(respectively $W^{k,q}_{per}(\Omega)$) are introduced as the closure of $\mathcal{C}^{\infty}_{per}(\Omega)$ in the norm $\|.\|_{L^{q}}$ (respectively
$\|\nabla^{k} .\|_{L^{q}}$) and  having zero mean value $\int_{\Omega} f(x) \: dx =0$, where $\|f \|_{L^{q}}^{q}= \int_{\Omega}|f (x)|^{q}\:dx $ and
$\|\nabla^{k} f \|_{L^{q}}^{q}=\int_{\Omega}|\nabla^{k} f (x)|^{q}\:dx$ are the respective norms.
Here $\nabla^{k}$ denotes the space gradient of order $k$. The spaces   $L^{q}_{per, div}(\Omega)$ (respectively $W^{k,q}_{per, div}(\Omega)$) stand for  the set of functions belonging to $L^{q}_{per}(\Omega)$
 (respectively $W^{k,q}_{per}(\Omega)$) with  zero divergence. Note that these spaces may also be defined as the closure of $\mathcal{C}^{\infty}_{per}(\Omega)$ with zero divergence and zero mean value with respect to the corresponding norms.\\
The scalar product in $L^{2}_{per}$ will be denoted  by $(.,.)$ and that in $W^{k,2}_{per}$ by $((.,.))_{k}$.
The scalar product of vectors $ u=(u_{1}, ..., u_{d})$ and $ v=(v_{1},..., v_{d})$ is denoted by $u . v= \sum_{i=1}^{d} u_{i}v_{i}$ while that of tensors $\mathbf{B}=(\mathrm{B}_{i j})_{1\leq i,j\leq d}$ and $\mathbf{D}=(\mathrm{D}_{i j})_{1\leq i,j\leq d}$ is denoted by  $\mathbf{B}:\mathbf{D}= \sum_{i,j=1}^{d}\mathrm{B}_{i j}\mathrm{D}_{i j}$.\\
Given a Banach space $X(\Omega)$ then $X^{*}(\Omega)$ stands for its dual space. We will not distinguish between  spaces of scalar, vector and tensor valued functions as one can easily make difference of sense between them. We  denote by $L^{q}(0; T;X(\Omega)) $ the usual  Bochner space consisting of functions which values  are in $X$ and are  $L^{q}$ time-integrable  over $(0,T)$ and by $L^{q'}(0,T;X^{*}(\Omega))$ its dual where $q'$ is the conjugate exponent of $q$ given by $q'=\frac{q}{q-1}$. On the other hand, $C([0,T];X(\Omega))$ stands for the space  of continuous in time  functions in $[0,T]$ and  with values in $X$. We also denote by $C_{w}(0,T;X(\Omega))$ the space of functions $u$ which are in $L^{\infty}(0,T;X(\Omega))$ and continue for almost every $t \in [0,T]$ for the weak topology of $X(\Omega)$.\\
Throughout the paper, if we denote by $c$ a positive constant with neither  any subscript nor superscript then $c$ is considered as a generic constant whose value can change from line to line in the inequalities and depends on the parameters in question but has no effect on the solution. On the other hand, we will denote in a bold character tensor functions and in the  usual one vector valued and scalar functions.\\
In the sequel of this section, we review some  technical results and classical  lemmas on compactness arguments and limiting process.\\
Since we are dealing with space periodic and divergence-free vector fields, one can enjoy some special identities that will be summarized in the following lemma.
\begin{lem}
Let $f: \Omega \longrightarrow \mathbb{R}$ and $u,\: v : \Omega \longrightarrow \mathbb{R}^{d}$ be periodic functions. Suppose that  $u$ and $v$ are divergence-free, then the following identities hold
\begin{eqnarray}
\label{diver1}\int_{\Omega} \nabla f \,.\, u \: dx &=& 0, \\
\label{diver2}\int_{\Omega} (u .\nabla) v \,. \,v \: dx &=& 0, \\
\label{diver3}\int_{\Omega} (u .\nabla) u\, . \,v \: dx &=& - \int_{\Omega} (u \otimes u) . \nabla v \: dx.
\end{eqnarray}
\end{lem}
\begin{proof}
The proof is based on the use of Green's formula, the divergence-free property and the fact that boundary terms vanish due to the periodicity. For details, one can see Lemma 2.9 in \cite{MNRR}.
\end{proof}

The famous Aubin-Lions lemma will play an important role in the sequel.
\begin{lem}[Aubin-Lions]\label{aubin lions}
Let $1<\alpha<\infty$, $1\leq \beta\leq \infty$ and $X_{0},X_{1}, X_{2}$ be  Banach reflexive separable spaces such that
$$ X_{0}\hookrightarrow\hookrightarrow X_{1} \;\; and \;\;  X_{1}\hookrightarrow X_{2}.$$
Then $$\left\{u \in L^{\alpha}(0,T;X_{0});\partial_{t} u \in L^{\beta}(0,T;X_{2})\right\}\hookrightarrow\hookrightarrow
L^{\alpha}(0,T;X_{1}).$$
\end{lem}
Here, the symbol $\hookrightarrow\hookrightarrow $ stands for the compact imbedding while $\hookrightarrow$ for the continuous one.
Lemma \ref{aubin lions} is a general version of the Aubin-Lions lemma valid under the fulfillment of  the  assumption $\beta=1$ and  is proved in \cite{roubicek} and \cite{Simon}, separately.
\begin{lem}[Vitali] \label{vitali}
Let $\Omega$ be a bounded domain in $\mathbb{R}^{d}$ and $f^{n}:\Omega\longrightarrow \mathbb{R}$ be a sequence of functions. Suppose that \\
- $\underset{n\rightarrow \infty}{\lim}f^{n}(x)$ exists and is finite for all $x\in \Omega$,\\
- for all $\varepsilon>0$, there exists $\delta>0$ such that
$$ \sup_{n\in \mathbb{N}} \int_{Q} |f^{n}(x)| \: dx < \varepsilon\;\;\; \forall Q \subset \Omega \;\;\;\mbox{with}\;\;\; |Q| <\delta.$$
Then
$$\underset{n\rightarrow \infty}{\lim}\int_{\Omega}f^{n}(x)\: dx = \int_{\Omega} \underset{n\rightarrow \infty}{\lim}f^{n}(x) \: dx.$$
\end{lem}
We would like to note that  the second assumption stated in the lemma above expresses the equi-integrability of the sequence $\{f^{n}\}$. Recall that if
$\{f^{n}\}$ is uniformly bounded in some Lebesgue spaces $L^{q}(\Omega)$ for $q>1$ then they are equi-integrable functions.

We finish this section by the Korn's inequality whose a proof can be found in \cite{MNRR} (see Theorem 1.10 page 196).
\begin{lem}[Korn's Inequality]\label{korn}
Let   $q\in (1,\infty) $, then there exists a positive constant $c_{k}$ depending only on $\Omega$ and $q$ such that for all $u\in W_{per}^{1,q}(\Omega)$
$$c_{k}\|\nabla u\|_{L^{q}} \leq  \|\mathbf{D}u\|_{L^{q}}.$$
\end{lem}

\section{Assumptions and main results}
Since we  consider a variable viscosity depending on the shear rate, this  makes the problem strongly  nonlinear and
brings some complications to its analysis. To deal with this situation, it seems naturally to impose further restrictions on the structure of the stress  tensor $\mathbf{S}$.\\
We denote by $\mathrm{S}_{i j}(\mathbf{D})$ the $i j $ entry of the matrix $\mathbf{S}(\mathbf{D})$. We shall assume that  the stress tensor function $\mathbf{S}:
\mathbb{M}^{d}_{sym} \longrightarrow \mathbb{M}^{d}_{sym}$ is continuously differentiable and satisfy $\mathbf{S}(\mathbf{0})= \mathbf{0}$.\\
In addition, we suppose that, for a given fixed parameter $r\in (2, \infty),$ there are positive constants  $c_{0}, c_{1}$ and $c_{2}$, such that for all $\mathbf{B}, \mathbf{D} \in
\mathbb{M}^{d}_{sym}$
\begin{eqnarray} \label{h1}
c_{0} (1+|\mathbf{D}|)^{r-2}|\mathbf{B}|^{2} \leq \frac{\partial \mathrm{S}_{i j}(\mathbf{D})}{\partial \mathrm{D}_{k l}} \mathrm{B}_{i j} \mathrm{B}_{k l}
\leq c_{1} (1+|\mathbf{D}|)^{r-2}|\mathbf{B}|^{2},
\end{eqnarray}
\begin{eqnarray}\label{h2}
\left|\frac{\partial \mathrm{S}_{i j}(\mathbf{D})}{\partial \mathrm{D}_{ k l}}\right| \leq c_{2} (1+|\mathbf{D}|)^{r-2}.
\end{eqnarray}
It is worth noticing that the nonlinearity of the tensor $\mathbf{S}$ has some useful properties that
follows from (\ref{h1}) and (\ref{h2}).
\begin{lem}\label{consq on visco}
Let $ r \in (2,\infty)$ be a fixed real number  and  the tensor $\mathbf{S}$ satisfy the assumptions  (\ref{h1}) and  (\ref{h2}). Then, there exist positive constants $c_{3}$ and  $c_{4}$  such that, for all $\mathbf{B},\;\mathbf{D} \in
\mathbb{M}^{d}_{sym}$,
\begin{eqnarray}
\label{coercivity}\mathbf{S} (\mathbf{D}):\mathbf{D} & \geq & c_{3} (1+ |\mathbf{D}|)^{r-2} |\mathbf{D}|^{2},\\
\label{growth}|\mathbf{S} (\mathbf{D})|  & \leq & c_{4} |\mathbf{D}| (1+|\mathbf{D}|)^{r-2},\\
\label{strict monotone} [\mathbf{S}(\mathbf{B}) - \mathbf{S}(\mathbf{D})] : [\mathbf{B} -\mathbf{D}] & \geq &
c_{3} |\mathbf{B} -\mathbf{D}|^{2} ( 1 + |\mathbf{B}| + | \mathbf{D}|)^{r-2}.
\end{eqnarray}
\end{lem}
\begin{proof}
See Lemma 1.19 page 198 in \cite{MNRR}.
\end{proof}
According to this lemma, the inequalities  (\ref{coercivity}) and  (\ref{growth})  express respectively  the $r$-coercivity and the polynomial growth of the tensor $ \mathbf{S}$ while (\ref{strict monotone}) describes the strict  monotonicity  of this operator. To be more precise, one can give a typical example of stress tensors, where the above structure is satisfied, for example
\begin{eqnarray}\label{typical}
\mathbf{S}(\mathbf{D})=(\mu_{0}+\mu_{1}|\mathbf{D}|)^{r-2} \mathbf{D},
\end{eqnarray}
where $\mu_{0}$ and $\mu_{1}$ are positive constants. The models (\ref{typical}) constitute a large family of stress tensors but do not include the proposed model ($\mu |\mathbf{D}|^{r-2}\mathbf{D}$) by Man  and Sun in  \cite{mansun} and later by Man \cite{man}. The main results of our paper are summarized in the following theorem.
\begin{thm}\label{Main result}
Let $r\in [3,\infty)$ and let  the stress tensor $\mathbf{S}$ verify  the assumptions  (\ref{h1}) and (\ref{h2}).
Assume that  the forcing term $f \in L^{2}(0,T; W^{1,2}_{per}(\Omega) )$ and  that $u_{0} \in W^{2,2}_{per,div}(\Omega)$.\\
Then there exists  a strong solution  $u $ to the problem (\ref{equa mere})-(\ref{per bound}) such that
\begin{eqnarray*}
& &u \in L^{\infty}(0,T;W_{per,\:div}^{2,2}(\Omega))\cap L^{r}(0,T;W_{per,\:div}^{1,r}(\Omega)) \\
\hspace*{10mm}& & u \in \mathcal{C}_{w}([0,T];W^{2,2}_{per,div}(\Omega)) \cap \mathcal{C}([0,T];W_{per,\:div}^{1,\sigma}(\Omega)),\;\;\;\mbox{with}\;\;\; \sigma \in [0,2),\hspace*{15mm}\\
& & \partial_{t} u  \in  L^{2}(0,T;L^{2}_{per,div}(\Omega)),\;\;\mbox{ if}\;\;\; d=2,\\
& &\partial_{t} u \in  L^{2}(0,T;L^{2}_{per,div}(\Omega))\cap L^{\frac{r}{r-1}}(0,T;W_{per,div}^{1,\frac{r}{r-1}}(\Omega)),\;\; \mbox{if}\;\;\; d=3,
\end{eqnarray*}
and satisfying the following weak formulation
\begin{eqnarray}\label{weak form}
& & \int_{0}^{T}(\partial_{t} u , \varphi)  \:d\tau - \alpha_{1} \int_{0}^{T} (\partial_{t} \Delta u , \varphi) \:d\tau  + \int_{0}^{T} ((u. \nabla ) u, \varphi) \:d\tau \nonumber \\
 &+& \alpha_{1} \int_{0}^{T}((u. \nabla ) \mathbf{D}\varphi, u)  \:d\tau  +  \int_{0}^{T} ( \mathbf{S} (\mathbf{D}u),  \mathbf{D} \varphi) \: d\tau\nonumber \\
& + & \alpha_{1} \int_{0}^{T} \int_{\Omega} \sum_{j, k=1}^{d}  \partial_{k}  u_{j} \partial_{i}  u_{j} \partial_{k} \varphi _{i} \:dx \: d\tau
 = \int_{0}^{T}(f , \varphi )  \: d\tau,
\end{eqnarray}
for all $\varphi \in  L^{2}(0,T; W_{per,\:div}^{2,2}(\Omega))\cap L^{r}(0,T;W^{1,r}_{per,div}(\Omega))$. Moreover, the solution $\boldsymbol{v}$ is unique and attain the initial conditions in the following sense
\begin{eqnarray}
\underset{t \rightarrow 0^{+}}{\lim} \|\boldsymbol{v}(t) -\boldsymbol{v}_{0}\|_{W^{1,2}}=0.
\end{eqnarray}
\end{thm}
\begin{rem}\quad\\
\vspace{-0.7cm}
\begin{itemize}
\item[i)] The principal reason for which we have considered the flow with shear thickening behavior and not with a shear thinning one comes from the following: when getting the energy inequality, we realize that $\nabla u \in L^{r}(0,T;L^{r}_{per}(\Omega))$ and since we have in our system terms with gradients
products, namely $\nabla u (\nabla u)^{t} $ and its transpose, then to ensure the space integrability of these terms (eventually in $L^{\frac{r}{2}}_{per}(\Omega)$) and  in accordance with the definition of Sobolev spaces, the parameter $\frac{r}{2}$ must be greater than $1$, and hence  $r\geq 2$.
\item[ii)] The existence of  the pressure can be deduced, with the help of the De Rham's theorem, from the weak formulation (\ref{weak form}) and one can assert that there is some distribution $p$ such that (\ref{equa mere}) holds in $[0,T] \times \mathbb{R}^{d}$.
\end{itemize}
\end{rem}
To the best of our knowledge, the question of existence of weak or strong solutions has not been
raised for the class of generalized second grade fluids. Furthermore, this work seems to be the first to prove global existence
and uniqueness of strong solution without restrictive smallness assumptions on the initial data.\\
Let us come back to some related results concerning  classical second grade fluids whose Cauchy stress is described by the formula  (\ref{classic cauchy tensor}). In \cite{galdisequeira}, Galdi and  Sequeira proved  existence and uniqueness of global in time classical solution when the initial data is very small in a bounded regular domain of $\mathbb{R}^{3}$ and subject to Dirichlet boundary conditions. Later, Le Roux in \cite{leRoux} extended these results under  suitable regularity and growth restrictions on the initial data with nonlinear partial slip boundary conditions in a bounded simply-connected domain. We also mention  that  Cioranescu and Ouazar \cite{cioraouazar} proved, for the Dirichlet boundary conditions case, existence and uniqueness of $W^{3,2}$
solutions that are global in two   dimensions of spaces and local for small data  in three dimensions. Further, Cioranescu and Girault \cite{cioragirault} improved the three dimensional result by showing global existence under further appropriate assumptions on the data.
On the other hand, Bresch and Lemoine \cite{breschlemoine} established existence and uniqueness  of $W^{2,r}$ ($r>3$) stationary solution for three dimensional bounded domain of class $\mathcal{C}^{2}$ with smallness  restrictions on the kinematic viscosity $\mu$ and the forcing term.\\
The first results concerning  unsteady incompressible flows of Navier-Stokes equations with shear rate dependent material coefficients go back to Ladyzhenskaya \cite{ladyzhenskaya} and  Lions \cite{lions}, who proved  for $r\geq \frac{3d}{d+2}$  the existence of weak solutions by using the monotone operator theory and compactness arguments. In the last two decades,  the mathematical analysis of
fluids with shear rate dependent viscosities have known a lot of relevant works dealing with existence of weak and strong  solutions and regularity results. One can cite some leaders in this field such as  M. Bul\'{\i}\v{c}ek, E. Feireisl, J. Frehse, J. M\'{a}lek, J. Ne\v{c}as, K.R. Rajagopal, M. R\.{u}\u{z}i\v{c}ka and many others  and  we refer to \cite{BMR}, \cite{malekraja07} and \cite{malekraja} where  overviews of theses results are established.\\

In this work,  we prove  global in time existence of strong solution (in $W^{1,r},\: r \geq 3$), for spatially periodic two and three-dimensional flows and  for a large class of fluids of second grade with  shear rate dependent viscosities. Moreover, in a second part  we prove uniqueness of such solution. All our results hold  without restricting on the size of the initial data.\\
The scheme of the proof is based on the use of  Galerkin method. First, we  construct the Galerkin approximations (based on the eigenfunctions of the Stokes operator) of the velocity and derive a priori $W^{1,2}$-estimates. But this is not enough to ensure the convergence of nonlinear terms. That is why we have to look for new estimates of higher derivatives order. More precisely, we derive the crucial estimates that are sufficient to establish the compactness of the velocity gradient for the approximations. \\
Unlike the classical fluids  of third grade, where in the $W^{1,2}$ and $W^{2,2}$ a priori estimates (established in a good way), one benefits from the presence of the term $ div \,( |\mathbf{D} u|^{2} \mathbf{D} u )$ which has a regularizing effect, this term is absent from the momentum equation of the classical second grade fluids. For this reason there are no existence results of $W^{2,2}$ solutions for the last cited fluid but only  $W^{3,2}$ solutions. In our case, one takes advantage from the coercivity of the stress tensor which will  allow us to get $W^{1,r}$  a priori estimates in a first step and help us in the absorption operation in order to obtain  $W^{2,2}$ estimates.

\section{Existence of strong solutions}
\subsection{Existence of approximate solutions}
In a first step, we  start with an approximate problem, whose solutions have sufficient regularity properties. This will be done by   performing a Galerkin scheme. So, let $\{\omega^{k}\}_{k=1}^{\infty}$ be  the set consisting of the $2 \pi$-periodic eigenvectors of the Stokes operator (denoted by $\mathrm{A}$) and  $\lambda_{k}$ be the corresponding eigenvalues. Note that $\int_{\Omega} \omega^{k} \: dx = 0.$
We set the Galerkin approximations $u^{n}$ of $u$ of  the form $\displaystyle u^{n}= \sum_{k=1}^{n}c_{k}^{n}(t)\omega^{k},$
where  the coefficients $c_{k}^{n}(t)$ solve the the $n$ coupled ordinary differential equations
\begin{eqnarray} \label{gal approx}
\begin{array}{rl}\displaystyle
(\partial_{t} v^{n}, \omega^{k})& = - ( (u^{n} \, .\,\nabla) v^{n}, \omega^{k})  - ( \mathbf{S}(\mathbf{D} u^{n}),\mathbf{D} \omega^{k} )
-  \displaystyle(\sum_{j=1}^{d} v_{j} \nabla u_{j}, \omega^{k}) \\
& +\:  ( f, \omega^{k}),\hspace*{1cm} v^{n} = u^{n} - \alpha_{1} \Delta u^{n},
\end{array}
\end{eqnarray}
for all $k=1,...,n$.\\
To close the system (\ref{gal approx}), we  prescribe the initial data  as
\begin{eqnarray}\label{initial approx}
u^{n}_{0} = \mathbb{P}^{n} u_{0} = \sum_{k=1}^{n} ( u_{0}, \omega^{k})\omega^{k}  \;\;\; \text{in} \;\;\Omega,
\end{eqnarray}
where $\mathbb{P}^{n}$ denotes  the orthogonal projection operator onto the subspace $\mathbb{H}_{n}:= span \{\omega^{1},...,\omega^{n}\}$  defined by $\mathbb{P}^{n} u = \sum_{k= 1 }^{n} ( u, \omega^{k}) \omega^{k}$. Furthermore, we have
\begin{eqnarray}
\lambda_{k} (u^{n}, \omega^{k}) = (\mathrm{A}u^{n}, \omega^{k}) = (\nabla u^{n}, \nabla \omega^{k}).
\end{eqnarray}
We would like to mention that $\mathbb{P}^{n} : W^{s,2}_{per,div}(\Omega)\longrightarrow \mathbb{H}_{n}$ is uniformly continuous for all $s \in [0,3]$ ( see \cite{MNRR} or \cite{ruzicka} for a proof.)\\
Due to the continuity of the right hand side of (\ref{gal approx}), the local existence   of a solution to (\ref{gal approx})-(\ref{initial approx}) on a short time interval $(0,t^{*})$ follows from the Carath\'{e}dory theory (see \cite{zeidler} for instance). In order to extend the solution to the whole time interval $[0, T]$, we need to show that the solution remains finite for
all positive times which consequently implies that $t^{*}= T$. To achieve this goal, we  will derive some  uniform estimates.\\
Multiplying the k-th equation in (\ref{gal approx}) by $c_{k}^{n}(t)$ and  summing over $k=1,...,n$,  we obtain the following energy equality
\begin{eqnarray}\label{apr estim}
\frac{1}{2} \frac{d}{dt} (\|u^{n}\|_{L^{2}}^{2} + \alpha_{1} \|\nabla u^{n}\|_{L^{2}}^{2})
+   \int_{\Omega} \mathbf{S}(\mathbf{D}u^{n}):\mathbf{D}u^{n} \:dx  = \int_{\Omega} f \:.\:u^{n} \:dx
\end{eqnarray}
All the other nonlinear terms vanish, it is due to the incompressibility constraint $div \, \omega^{k} =0$, see the identities (\ref{diver1}) and (\ref{diver2}).\\
Now integrating  the identity (\ref{apr estim})  over time on $[0,t]$ and  using (\ref{coercivity}) and the Korn's inequality, we deduce that
\begin{eqnarray}\label{apr estim 1}
\|u^{n}(t)\|_{L^{2}}^{2} + \alpha_{1} \|\nabla u^{n}(t) \|_{L^{2}}^{2}
+ c_{3} c_{k} \int_{0}^{t} \|\nabla  u^{n}\|_{L^{r}}^{r} \:d\tau\nonumber\hspace*{4cm} \\
\leq \|u_{0}^{n}\|_{L^{2}}^{2}  + \alpha_{1}\|\nabla u_{0}^{n}\|_{L^{2}}^{2} + c(\varepsilon) \int_{0}^{t} \|f\|_{L^{r^{'}}}^{r^{'}} \: d\tau +  \varepsilon c_{s}\int_{0}^{t} \| \nabla u \|_{L^{r}}^{r} \: d\tau
\end{eqnarray}
where $c_{s}$ is the Sobolev constant of the embedding $W^{1,r}_{per}(\Omega)\hookrightarrow L^{r}_{per}(\Omega)$. \\
Choosing $\varepsilon$ small enough, we obtain the following uniform bounds  for all $t \in (0,T)$
\begin{eqnarray}
\label{bound1} \{u^{n}\} \;\;\;&\;\;\mbox{ is uniformly bounded in}\;\; &\;\;\;\;\; L^{\infty}(0,t; W^{1,2}_{per, div}(\Omega))\\
\label{bound2} \{u^{n}\} \;\;\;& \;\;\mbox{is uniformly bounded  in}\;\; &\;\;\;\;\; L^{r}(0,t; W^{1,r}_{per, div}(\Omega)).
\end{eqnarray}
From (\ref{bound1})  it follows that
\begin{eqnarray}
 | c_{k}^{n}(t)|^{2}<\infty \;\;\; \mbox{for all}\;\;\; t\in [0,T], \;\;\;\forall \;\;\; k=1,...,n.
\end{eqnarray}
Taking into account the continuity of $c_{k}^{n}$  on $[0,t^{*}[$, one can shift $t^{*}$ to $T$. Thus we realize that  there exists $u$ such that for a certain subsequence (still denoted by $\{u^{n}\}$), we have
\begin{eqnarray}
\label{30} u^{n} & \rightharpoonup & u \quad\mbox{*weakly in}\quad L^{\infty}(0,T;W^{1,2}_{per,div}(\Omega)),\\
\label{31} u^{n} & \rightharpoonup & u \quad\mbox{weakly in}\quad L^{r}(0,T;W_{per,div}^{1,r}(\Omega)).
\end{eqnarray}
Moreover, using the growth property (\ref{growth}), we infer that  the sequence $\{\mathbf{S}(\mathbf{D} u^{n})\}$ is uniformly bounded in
$L^{r^{'}}(0,T;L_{per,div}^{r^{'}}(\Omega))$ where $r^{'}=\frac{r}{r-1}$ is the conjugate exponent of $r$.
Therefore, we can deduce that, up to subsequence extraction, there exists a tensor $\boldsymbol{\chi}$ for which
\begin{eqnarray}
\label{321}\mathbf{S}(\mathbf{D} u^{n}) \;\rightharpoonup \; \boldsymbol{\chi}
\quad\mbox{weakly in}\quad L^{r^{'}}(0,T;L_{per,div}^{r^{'}}(\Omega)).
\end{eqnarray}
The main difficulty in the limiting process consists in  the terms which are  nonlinear in  $\nabla u^{n}$ and more precisely those with  product of spatial derivative of first order. So  we are urged to obtain stronger regularity results on the sequence $\{u^{n}\}$ in order to have compactness of the sequence $\{\nabla u^{n}\}$, and thus consequently the pointwise convergence. The next paragraph will focus on.
\subsection{Compactness and limiting process}
The usual method  to establish strong solutions for similar problems consists in considering   $-\Delta u^{k}$ as a test function in  the equation (\ref{gal approx}) and trying to establish  a priori estimates for the spatial derivative of second order of the velocity field. This is only allowed if  $-\Delta u^{k}$ is sufficiently regular ( say in $L^{2}_{per, div}(\Omega)$), but until now  this is not verified. Fortunately, the periodic boundary setting enable us this operation since $A u^{k} = - \Delta u^{k}$ and the task will simplify to multiplying the equation in interest by the corresponding eigenvalues $\{\lambda^{k} \}$ of the Stokes operator.\\
Multiplying the equations (\ref{gal approx}) by $\lambda_{k} c_{k}^{n}(t)$ and  summing the resulted equations over $k=1,...,n$, we obtain by means of integrations by parts
\begin{eqnarray}\label{energy eq 2}
\frac{1}{2} \frac{d}{dt} (\|\nabla u^{n}\|_{L^{2}}^{2} + \alpha_{1} \|\nabla^{2} u^{n}\|_{L^{2}}^{2})
+ (div( \mathbf{S}(\mathbf{D} u^{n})), \Delta u^{n})\nonumber \\
=   ( (u^{n} \, .\,\nabla) v^{n}, \Delta u^{n}) +  (\sum_{j=1}^{d} v_{j} \nabla u_{j}, \Delta u^{n}) -  ( f, \Delta u^{n}).
\end{eqnarray}
In the following, we will estimate appropriately the different terms appearing in (\ref{energy eq 2}). Using (\ref{h1}) and (\ref{coercivity}) we get from the stress tensor
\begin{eqnarray*}
(div( \mathbf{S}(\mathbf{D}u^{n})), \Delta u^{n}) &=& (\partial_{\mathbf{D}} \mathbf{S}(\mathbf{D}u^{n}) \mathbf{D} \nabla u^{n}, \mathbf{D}\nabla u^{n})\\
&=& \sum_{i,j,k, l=1}^{d}\int_{\Omega}\frac{\partial \mathbf{S}_{i j}(\mathbf{D} u^{n})} {\partial \mathrm{D}_{k l}}\mathrm{D}_{k l} (\nabla u^{n})
 \mathrm{D}_{i j} (\nabla u^{n})\: dx\nonumber\\
& \geq &  c_{0} \int_{\Omega} (1+|\mathbf{D}u^{n}|)^{r-2} |\mathbf{D}\nabla u^{n}|^{2} \: dx = c_{0} \mathcal{I}_{r} (u^{n})\\
& \geq &  c_{0} \int_{\Omega} |\mathbf{D} u^{n}|^{r-2} |\mathbf{D}\nabla u^{n}|^{2} \: dx + c_{0} \int_{\Omega} |\mathbf{D}\nabla u^{n}|^{2} \: dx.
\end{eqnarray*}
By Korn's inequality (\ref{korn}), we deduce that
\begin{eqnarray}\label{433}
(div( \mathbf{S}(\mathbf{D}u^{n})), \Delta u^{n}) &\geq & c_{0} \mathcal{I}_{r} (u^{n})\nonumber \\
&\geq &   c_{0} \int_{\Omega} |\mathbf{D} u^{n}|^{r-2} |\mathbf{D}\nabla u^{n}|^{2} \: dx + c_{0}c_{k} \|\nabla^{2} u^{n}\|_{L^{2}}^{2}.
\end{eqnarray}
Furthermore, we recall that the quantity
\begin{eqnarray} \label{zero}
( (u^{n}.\nabla) u^{n}, \Delta u^{n}) = \sum_{i,j, k=1}^{d} \int_{\Omega} \partial_{k} u_{j}^{n} \partial_{j} u_{i}^{n} \partial_{k} u_{i}^{n} \: dx,
\end{eqnarray}
vanishes in the two  dimensional space-periodic setting  due to the fact that  $\partial_{1} u_{1}^{n}= - \partial_{2}u_{2}^{n}$. Using integration by parts and the divergence-free condition of $u^{n}$, we see from (\ref{exp 1}) and (\ref{diver1}) that we have
\begin{eqnarray}\label{exp 11}
((u^{n} \, .\,\nabla) v^{n}, \Delta u^{n} ) &=& ((u^{n} \, .\,\nabla) u^{n}, \Delta u^{n} ) - \alpha_{1} ((u^{n} \, .\,\nabla) \Delta u^{n}, \Delta u^{n} )  \\
&=&  \sum_{i,j, k=1}^{d}\int_{\Omega} \partial_{k} u_{j}^{n} \partial_{j} u_{i}^{n} \partial_{k} u_{i}^{n} \: dx - \alpha_{1} \sum_{i, j=1}^{d}\int_{\Omega} u_{j}^{n} \partial_{j} \Delta u_{i}^{n}  \Delta u_{i}^{n}) \nonumber \\
 &= &\sum_{i,j, k=1}^{d}\int_{\Omega} \partial_{k} u_{j}^{n} \partial_{j} u_{i}^{n} \partial_{k} u_{i}^{n} \: dx  - \frac{\alpha_{1}}{2} \int_{\Omega} u_{j}^{n} \partial_{j} |\Delta u_{i}^{n}|^{2} \: dx \nonumber\\
& \leq  & \|\nabla u^{n}\|_{L^{3}}^{3} \leq c \|\nabla u^{n}\|_{L^{r}} \|\nabla^{2} u^{n}\|_{L^{2}}^{2}.
\end{eqnarray}
On the other hand, one can deduce from (\ref{exp 2}) that
\begin{eqnarray}\label{446}
(\sum_{j=1}^{d} v_{j}^{n} \nabla u_{j}^{n},  \Delta u^{n})  &=& \alpha_{1} \sum_{i, j, k=1}^{d} \int_{\Omega} \partial_{k}( \partial_{k}  u_{j}^{n} \partial_{i}  u_{j}^{n}) \Delta u_{i}^{n}\: dx\nonumber\\
&=& \alpha_{1} \sum_{i, j = 1}^{d} \int_{\Omega} \Delta u_{j}^{n} \partial_{i}  u_{j}^{n} \Delta u_{i}^{n}\: dx + \alpha_{1} \sum_{i, j, k=1}^{d} \int_{\Omega}  \partial_{k}  u_{j}^{n} \partial_{k i}  u_{j}^{n} \Delta u_{i}^{n}\: dx\nonumber\\
&=& \alpha_{1} \sum_{i, j = 1}^{d} \int_{\Omega} \Delta u_{j}^{n} \left[ \frac{1}{2} (\partial_{i}  u_{j}^{n} + \partial_{j} u_{i}^{n} )\right] \Delta u_{i}^{n}\: dx \nonumber \\
&+ &\frac{\alpha_{1}}{2} \sum_{i=1}^{d} \int_{\Omega} (\partial_{i}|\nabla u^{n}|^{2}) \Delta u_{i}^{n}\: dx\nonumber\\
& \leq &  \alpha_{1}  \int_{\Omega} |\Delta u^{n}|^{2} |\mathbf{D} u^{n}| \: dx
\leq   \alpha_{1}  \int_{\Omega} |\mathbf{D} \nabla u^{n}|^{2} |\mathbf{D} u^{n}| \: dx,
\end{eqnarray}
where in the third line of the last estimate  we have used the symmetry of the scripts $i$ and $j$.
If $r > 3$ then by means of the  Young's inequality, one have
\begin{eqnarray}\label{r3}
|\mathbf{D} u^{n}|= |\mathbf{D} u^{n}| 1 \leq \varepsilon |\mathbf{D} u^{n}|^{r-2}  + c (\varepsilon) 1^{\frac{r-2}{r-3}},
\end{eqnarray}
which yields that
\begin{eqnarray}\label{447}
(\sum_{j=1}^{d} v_{j}^{n} \nabla u_{j}^{n},  \Delta u^{n})  \leq  \alpha_{1}\varepsilon  \int_{\Omega} |\mathbf{D} \nabla u^{n}|^{2} |\mathbf{D} u^{n}|^{r-2}\: dx +  \alpha_{1} c(\varepsilon ) \int_{\Omega} |\mathbf{D} \nabla u^{n}|^{2} \: dx.
\end{eqnarray}

Putting all the estimates (\ref{433}), (\ref{446}) and  (\ref{447}) together,  the energy inequality (\ref{energy eq 2}) becomes
\begin{eqnarray}\label{energy ineq 22}
\frac{1}{2} \frac{d}{dt} \|\nabla u^{n}\|_{L^{2}}^{2} + \frac{\alpha_{1}}{2} \frac{d}{dt} \|\nabla^{2} u^{n}\|_{L^{2}}^{2}
 +   c_{0} \int_{\Omega} |\mathbf{D} u^{n}|^{r-2} |\mathbf{D}\nabla u^{n}|^{2} \: dx + c_{0}c_{k} \|\nabla^{2} u^{n}\|_{L^{2}}^{2}\nonumber \\
 \leq \alpha_{1} \varepsilon \int_{\Omega} |\mathbf{D} u^{n}|^{r-2}|\mathbf{D}\nabla u^{n}|^{2} \: dx + c(\alpha_{1},\varepsilon) (\|\nabla u^{n}\|_{L^{r}} + 1 + \|f\|_{L^{2}}) \|\nabla^{2} u^{n}\|_{L^{2}}^{2}.
\end{eqnarray}
Choosing $\varepsilon <<1 $ small enough in such a way that $\alpha_{1} \varepsilon < c_{0}$ one can absorb the first term of the right hand side of (\ref{energy ineq 22}) and find that
\begin{eqnarray}\label{energy ineq 23}
\frac{1}{2} \frac{d}{dt} \|\nabla u^{n}\|_{L^{2}}^{2} + \frac{\alpha_{1}}{2} \frac{d}{dt} \|\nabla^{2} u^{n}\|_{L^{2}}^{2}
 +   c \int_{\Omega} |\mathbf{D} u^{n}|^{r-2} |\mathbf{D}\nabla u^{n}|^{2} \: dx \nonumber \\
 + c_{0}c_{k} \|\nabla^{2} u^{n}\|_{L^{2}}^{2}  \leq \tilde{c} \|\nabla u^{n}\|_{L^{r}}  \|\nabla^{2} u^{n}\|_{L^{2}}^{2},\hspace*{2cm}
\end{eqnarray}
where $c$ and $\tilde{c} $ are  some positive fixed constants. Note that for $r=3$, we do not need to use (\ref{r3}) and inequality (\ref{energy ineq 23})  remains valid. Consequently, since $\|\nabla u^{n}\|_{L^{r}}+ \|f\|_{L^{2}} $ is uniformly bounded  in $L^{1}(0,T)$,  the Gronwall's lemma enables us to conclude for both cases $d=2$ and $d=3$ that
\begin{eqnarray}\label{second deriv bound}
\{u^{n}\} \;\; &\quad\mbox{is uniformly bounded in}\quad &\;\; L^{\infty}(0,T;W^{2,2}_{per,div}(\Omega)),
\end{eqnarray}
and that for all $n$
\begin{eqnarray}\label{control i r}
\int_{0}^{T}\int_{\Omega} (1 +|\mathbf{D}u^{n}|)^{r-2} |\mathbf{D}\nabla u^{n}|^{2} \: dx \: d\tau \leq c.
\end{eqnarray}
Thus,  we have
\begin{eqnarray}
\label{40} u^{n}&\rightharpoonup& u\quad\mbox{*weakly in}\quad L^{\infty}(0,T;W^{2,2}_{per,div}(\Omega)),\\
\label{41} u^{n}&\rightharpoonup& u\quad\mbox{weakly in}\quad L^{q}(0,T;W_{per,div}^{2,2}(\Omega)),\;\;\; \forall \; q \in [1, \infty).
\end{eqnarray}
In order to prove a compactness  for  the velocity field in some Sovolev spaces, we need uniform estimates for the time derivative $\partial_{t} u^{n}$. To do this, observe that for test function
\begin{eqnarray*}
\varphi &\in& L^{2}(0,T;W_{per,div}^{2,2}(\Omega)) \cap L^{r}(0,T;W_{per,div}^{1,r}(\Omega)),\;\;\;\mbox{if}\;\; \; d=3\\
\varphi &\in& L^{2}(0,T;W_{per,div}^{2,2}(\Omega)) ,\;\;\;\mbox{if}\;\;\; d=2,\\
\end{eqnarray*}
 we have
\begin{eqnarray}\label{time bound}
\begin{array}{rl}
((\partial_{t} v^{n} , \varphi))_{2}= ((\partial_{t} v^{n}, \mathbb{P}^{n}\varphi))_{2}
&=- ( (u^{n} \, .\,\nabla) v^{n}, \mathbb{P}^{n}\varphi)  +  ( div (\mathbf{S}(\mathbf{D} u^{n})),\mathbb{P}^{n}\varphi )\\
& \displaystyle -  (\sum_{j=1}^{d} v_{j} \nabla u_{j}, \mathbb{P}^{n} \varphi)
 +  ( f, \mathbb{P}^{n}\varphi).
\end{array}
\end{eqnarray}
Let us now estimate the right hand side of  (\ref{time bound}). Recall that the projection operator $\mathbb{P}^{n}$ is continuous from $W^{q,2}_{per,div}(\Omega)$ to $\mathbb{H}_{n}$ for all $q \in [0,3]$. Recall also the usual Sobolev embeddings \cite{adams}
\begin{eqnarray}\label{sob2}
W_{per}^{1,2}(\Omega)\hookrightarrow  L_{per}^{q}(\Omega),& &\;\;\;\forall\;q \in [1,\infty),\quad\mbox{if}\;\;\; d=2\\
\label{sob3}W_{per}^{1,2}(\Omega)\hookrightarrow L_{per}^{q}(\Omega),& &\;\;\;\forall\;q \in [1,6],\quad\mbox{if}\;\;\; d=3.
\end{eqnarray}
Moreover,  it is obvious that
\begin{eqnarray*}
\frac{2r}{r-2} \in \begin{cases}\displaystyle
(2, \infty), & \quad\mbox{if}\;\;\; r \in (2, \infty), \\
 (2, 6],& \quad\mbox{if}\;\;\; r \in [3,\infty).
  \end{cases}
\end{eqnarray*}
Therefore, by (\ref{exp 1}) and the  H\"{o}lder's inequality, we obtain
\begin{eqnarray*}
\displaystyle\int_{0}^{T}( (u^{n} \, .\,\nabla) v^{n}, \mathbb{P}^{n}\varphi) \: d \tau
& \leq & \int_{0}^{T} \|u^{n} \|_{L^{\infty}} \|\nabla u^{n} \|_{L^{2}} (\|\mathbb{P}^{n} \varphi \|_{L^{2}} + \alpha_{1} \|\nabla^{2} \mathbb{P}^{n} \varphi \|_{L^{2}}) \: d\tau \\
&+& \alpha_{1} \int_{0}^{T} \|\nabla u^{n} \|_{L^{r}} \|\nabla u^{n} \|_{L^{2}}  \|\nabla \mathbb{P}^{n} \varphi \|_{L^{\frac{2r}{r-2}}}\:d \tau\\
& \leq & c \|u^{n} \|_{L^{\infty}(W^{2,2})}  \|\nabla u^{n} \|_{L^{2}(L^{r})}  \| \mathbb{P}^{n} \varphi \|_{L^{2}( W^{2,2})}\leq c.
\end{eqnarray*}
Similarly, using (\ref{exp 2}) and taking into account the divergence-free condition, we get
\begin{eqnarray*}
\displaystyle \int_{0}^{T}(\sum_{j=1}^{d} v_{j} \nabla u_{j}, \mathbb{P}^{n} \varphi) \: d \tau
& \leq &\alpha_{1} \int_{0}^{T} \|\nabla u^{n} \|_{L^{r}} \|\nabla u^{n} \|_{L^{2}}  \|\nabla \mathbb{P}^{n} \varphi \|_{L^{\frac{2r}{r-2}}}\:d \tau\\
& \leq & c \|\nabla u^{n} \|_{L^{\infty}(L^{2})} \|u^{n} \|_{L^{2}(W^{1,r})} \| \mathbb{P}^{n} \varphi \|_{L^{2}( W^{2,2})}\leq c.
\end{eqnarray*}
To conclude, we consider now the stress tensor  term.
Thanks to the growth property (\ref{growth}) we have
\begin{eqnarray*}
\displaystyle\int_{0}^{T}(div (\mathbf{S}(\mathbf{D} u^{n})), \mathbb{P}^{n}\varphi) \: d \tau &=& \int_{0}^{T}(\mathbf{S}(\mathbf{D} u^{n}), \mathbf{D}\mathbb{P}^{n}\varphi) \: d \tau\\
&\leq & c_{4}  \int_{0}^{T} \int_{\Omega} |  (1 + |\mathbf{D} u^{n}|)^{r-1} |\mathbf{D} \mathbb{P}^{n} \varphi |\: dx  \:d \tau\\
&\leq & c \displaystyle \int_{0}^{T} \int_{\Omega}|\mathbf{D} \mathbb{P}^{n} \varphi | + |\mathbf{D} u^{n}|^{r-1} |\mathbf{D}(\mathbb{P}^{n} \varphi)|\: dx\:d \tau\\
&\leq & c \displaystyle \|\varphi \|_{L^{1}(W^{1,2})} +  \int_{0}^{T} \|\mathbf{D} u^{n}\|_{L^{r}}^{r-1}\|\mathbf{D}(\mathbb{P}^{n} \varphi)\|_{L^{r}}\:d \tau\\
&\leq & c \displaystyle \|\varphi \|_{L^{1}(W^{1,2})} +  \|u^{n}\|_{L^{r} (W^{1,r})}^{r-1}\|\varphi\|_{L^{r}(W^{1,r})}\leq c.
 \end{eqnarray*}
In dimension two we can improve, thanks to (\ref{sob2}), the last estimate as follows
\begin{eqnarray*}
\displaystyle\int_{0}^{T}(div (\mathbf{S}(\mathbf{D} u^{n})), \mathbb{P}^{n}\varphi) \: d \tau \leq c \|\varphi \|_{L^{1}(W^{1,2})} +  \|u^{n}\|_{L^{\infty} (W^{2,2})}^{r-1}\|\varphi\|_{L^{1}(W^{2,2})}\leq c.
\end{eqnarray*}
Now, denoting by $X^{2,r}=L^{2}(0,T;(W^{2,2}_{per,div}(\Omega))^{*})\cap L^{\frac{r}{r-1}}(0,T;W_{per,div}^{1,r}(\Omega)^{*})$, and keeping in mind the last estimates, we infer that
\begin{eqnarray}\label{timebound v}
\begin{array}{rl}
\{\partial_{t} v^{n}\} \;&\mbox{is uniformly bounded in} \quad L^{2}(0,T;(W^{2,2}_{per,div}(\Omega))^{*}),\;\; if\; d=2,\\
\{\partial_{t} v^{n}\} \; &\mbox{is uniformly bounded in}\quad X^{2,r} \;\; if\; d=3,
\end{array}
\end{eqnarray}
and hence by the classical elliptic regularity we deduce that
\begin{eqnarray}\label{timebound u}
\begin{array}{rl}\{\partial_{t} u^{n}\} &\;\mbox{is uniformly bounded in}\quad L^{2}(0,T;L^{2}_{per,div}(\Omega)),\; if\; d=2,\\
\{\partial_{t} u^{n}\} \; &\mbox{is uniformly bounded in}\quad  Y^{2,r}\; if\; d=3,
\end{array}
\end{eqnarray}
where $Y^{2,r}= L^{2}(0,T;L^{2}_{per,div}(\Omega))\cap L^{\frac{r}{r-1}}(0,T;W_{per,div}^{1,\frac{r}{r-1}}(\Omega))$. Consequently, we infer that
\begin{eqnarray}\label{weaklim time u}
\begin{array}{rl}\{\partial_{t} u^{n}\} &\;\rightharpoonup\;  \partial_{t} u\quad\mbox{in}\;\;\;\; L^{2}(0,T;L^{2}_{per,div}(\Omega)),\quad\mbox{if}\;\;\;d=2,\\
\{\partial_{t} u^{n}\} \;\; &\;\rightharpoonup \;\partial_{t} u\quad\mbox{in}\;\;\;\; L^{2}(0,T;L^{2}_{per,div}(\Omega))\cap L^{\frac{r}{r-1}}(0,T;W_{per,div}^{1,\frac{r}{r-1}}(\Omega))\quad\mbox{if}\;\;\; d=3.
\end{array}
\end{eqnarray}
Since the power $r$ lies in $[3,\infty)$ then $\frac{r}{r-1} \in (1,\frac{3}{2}]$. So, by the Aubin-Lions lemma \ref{aubin lions}, (\ref{bound2}) and (\ref{weaklim time u}), we deduce that
\begin{eqnarray}\label{61}
\label{35}u^{n} \rightarrow u \quad\mbox{strongly in}\quad  L^{r}(0,T;W_{per, div}^{1,r}(\Omega)),\;\;\;
\end{eqnarray}
which implies that
\begin{equation}
\label{36}u^{n}\rightarrow u \quad\mbox{almost everywhere in}\quad [0,T]\times \Omega,
\end{equation}
\begin{equation}
\label{44}\nabla u^{n}\rightarrow \nabla u \quad\mbox{almost everywhere in}\quad [0,T]\times \Omega,
\end{equation}
and by symmetry
\begin{eqnarray}
\label{448}\mathbf{D} u^{n}\rightarrow \mathbf{D} u\quad&\mbox{almost everywhere in}&\quad [0,T]\times \Omega.
\end{eqnarray}
Since the stress tensor $\mathbf{S}$ is continuous
\begin{eqnarray}
\label{449}\mathbf{S}(\mathbf{D} u^{n})\rightarrow \mathbf{S}(\mathbf{D} u) \quad&\mbox{almost everywhere in}&\quad [0,T]\times \Omega.
\end{eqnarray}
Consequently, by means of Vitali's lemma we deduce that $ \boldsymbol{\chi} = \mathbf{S}(\mathbf{D}u)$  and one have
\begin{eqnarray}
\label{500}\mathbf{S}(\mathbf{D} u^{n})\rightarrow \mathbf{S}(\mathbf{D} u) \quad &\mbox{weakly in}&\quad L^{r^{'}}(0,T;L_{per,div}^{r^{'}}(\Omega)),
\end{eqnarray}
and almost everywhere in $ [0,T]\times \Omega$.\\
Now we outline the passage to the limit in the weak formulation (\ref{weak form}). Using (\ref{timebound v}) and  (\ref{timebound u}), we can pass to the limit in the terms corresponding to the time derivative for $\Delta u$ and $u$, respectively. On the other hand,  the limiting process in the convective terms involving (\ref{exp 1}) and (\ref{exp 2}) and the stress tensor are  ensured by (\ref{61}), (\ref{35}) and (\ref{500}).\\
Finally, the  continuity property follows from the fact that $u \in L^{2}(0,T;L^{2}_{per, div}(\Omega))$ and  $\partial_{t} u \in L^{2}(0,T;L^{2}_{per,div}(\Omega))$ which implies that
$u \in \mathcal{C}([0,T];L^{2}_{per, div}(\Omega))$. To improve this regularity property, we see that for a fixed time $t_{0} \in (0,T)$ we have the following interpolation inequality
$$ \| u (t,.) - u(t_{0}, .\|_{W^{\sigma,2}} \leq c \| u (t,.) - u(t_{0}, .\|_{L^{2}}^{\frac{\sigma}{2}} \| u (t,.) - u(t_{0}, .)\|_{W^{2,2}},\;\;\; \forall \; \sigma \in [0,2),$$
and therefore we deduce   the  strong continuity in $W^{\sigma,2}_{per,div}(\Omega)$.\\
Concerning the weak continuity, it is easy to show
$$(u(t), \psi)  + \alpha_{1} (\nabla u (t) ,\nabla \psi) \underset{t \rightarrow t_{0}}{\longrightarrow}  ( u(t_{0}), \psi) + \alpha_{1} (\nabla u (t_{0}) ,\nabla \psi), $$
for\; all \;$ t_{0} \in [0,T]$ and $ \psi \in W^{2,2}_{per,div}(\Omega).$\\
To finish this section, we can easily handle the attainment of the initial data. For more details one can consult paragraph 3.10 in \cite{malekraja}.
\section{Uniqueness of strong solution}
Consider  $u$  and $\bar{u}$ (with initial data $u_{0}$ and $\bar{u}_{0}$ respectively)  two strong solutions to the
problem consisting of equations (\ref{div oh})-(\ref{per bound}) as defined in Theorem \ref{Main result}. We set $w:= u -  \bar{u}$ their difference. Subtracting
the equations   relatively to $u$ and $\bar{u}$ we obtain the following system
\begin{eqnarray} \label{equa w}
\displaystyle
\partial_{t} w - \alpha_{1}\partial_{t} \Delta w + (u \, .\,\nabla) (u -\alpha_{1} \Delta u) - (\bar{u} \, .\,\nabla) (\bar{u} -\alpha_{1} \Delta \bar{u})\hspace*{3cm}  \\
- div ( \mathbf{S}(\mathbf{D} u) -  \mathbf{S}(\mathbf{D} \bar{u}))
+ \sum_{j=1}^{d} (u_{j} - \alpha_{1} \Delta u_{j}) \nabla u_{j} - \sum_{j=1}^{d} (\bar{u}_{j} - \alpha_{1} \Delta \bar{u}_{j}) \nabla \bar{u}_{j} =0.\nonumber
\end{eqnarray}
Multiplying (\ref{equa w}) by $w$ and integrating over $\Omega$,  we obtain by (\ref{strict monotone}) the following energy inequality
\begin{eqnarray}\label{uniq}
\frac{1}{2} \frac{d}{d t} \|w(t)\|_{L^{2}}^{2} + \frac{\alpha_{1}}{2} \frac{d}{d t}\|\nabla w (t)\|_{L^{2}}^{2}
+ \int_{\Omega} (1 + |\mathbf{D} u| + |\mathbf{D} \bar{u} |)^{r-2} |\mathbf{D} u - \mathbf{D} \bar{u} |^{2}\:dx \nonumber \\
\leq - \int_{\Omega} ((u.\nabla) u- (\bar{u}.\nabla) \bar{u}) . w  \:dx - \alpha_{1}\int_{\Omega} ((u.\nabla) \mathbf{D}u  -(\bar{u}.\nabla)\mathbf{D}\bar{u})
:\mathbf{D} w \:dx \nonumber\\
 -\alpha_{1}\int_{\Omega} \sum_{j=1}^{d} (u_{j}- \alpha_{1} \Delta u _{j} ) \nabla u_{j} -  (\bar{u}_{j}- \alpha_{1} \Delta \bar{u} _{j} ) \nabla \bar{u}_{j})\: . \: w \: dx :=\mathcal{I}_{1} + \mathcal{I}_{2} + \mathcal{I}_{3}.
\end{eqnarray}
The first term of the right hand side of (\ref{uniq}) can be handled in the following way
\begin{eqnarray}
\mathcal{I}_{1} =  -\displaystyle \int_{\Omega} (w .\nabla)  u \:. \: w + (\bar{u}. \nabla)  w . w \:dx  & = & -\displaystyle \int_{\Omega} (w .\nabla)  u \:. \: w \; dx \nonumber\\
& \leq &  \| w\|_{L^{4}}^{2} \| \nabla u\|_{L^{2}}  \leq   \| \nabla u\|_{L^{2}} \|\nabla w\|_{L^{2}}^{2}.\;\;\;\;\;
\end{eqnarray}
Note that $\int_{\Omega}(\bar{u}. \nabla)  w . w \:dx = 0$ ( see (\ref{diver2})) due to the fact that $w$ is a divergence-free vector field. Thanks to (\ref{diver2}) and H\"{o}lder's inequality, we have
\begin{eqnarray*}
\mathcal{I}_{2}&=& - \alpha_{1}\int_{\Omega} ((u .\nabla) \mathbf{D}w + (w.\nabla)\mathbf{D}\bar{u})
:\mathbf{D}w \:dx \\
&=& - \alpha_{1} \int_{\Omega}  (w.\nabla)\mathbf{D}\bar{u})
:\mathbf{D}w \:dx \\
 & \leq & \alpha_{1} \int_{\Omega}  |w| |\mathbf{D}\nabla \bar{u}||\nabla w| \:dx\\
& \leq &  \alpha_{1} \int_{\Omega}  |w| ( 1 + | \mathbf{D} \bar{u}|)^{\frac{r-2}{2}}|\mathbf{D}\nabla \bar{u}| ( 1 + | \mathbf{D} \bar{u}|)^{\frac{2-r}{2}} |\nabla w| \:dx\\
&\leq&    \alpha_{1} \| w \|_{L^{\infty}} \left(\int_{\Omega}  ( 1 + | \mathbf{D} \bar{u}|)^{r-2}|\mathbf{D}\nabla \bar{u}|^{2} \:dx\right) \|\nabla w \|_{L^{2}}.
\end{eqnarray*}
Next, we deal with  the  last term in (\ref{uniq}). Using (\ref{exp 2}), we have
\begin{eqnarray}\label{exp 22}
\mathcal{I}_{3}&:= & \int_{\Omega} \sum_{j=1}^{d} (u_{j}- \alpha_{1} \Delta u _{j} ) \nabla u_{j} -  (\bar{u}_{j}- \alpha_{1} \Delta \bar{u} _{j} ) \nabla \bar{u}_{j})\: . \: w \: dx\nonumber \\
&=& -  \alpha_{1}\int_{\Omega}   [\sum_{i, j, k =1}^{d}\partial_{k}( \partial_{k}  u_{j} \partial_{i}  u_{j})- \partial_{k}( \partial_{k}  \bar{u}_{j} \partial_{i}  \bar{u}_{j})]  w_{i} \: dx\\
&=&   \alpha_{1}\int_{\Omega}   \sum_{i, j, k =1}^{d}( \partial_{k}  w_{j} \partial_{i}  u_{j} \partial_{k} w_{i}  +  \partial_{k}  \bar{u}_{j} \partial_{i}  w_{j}\partial_{k} w_{i} \: dx\nonumber \\ \;\;\;
\end{eqnarray}
By symmetry and integration by parts, we get
\begin{eqnarray}\label{exp 23}
\mathcal{I}_{3}&=&   -\alpha_{1}\int_{\Omega}   \sum_{i, j, k =1}^{d} \partial_{i k}  w_{j} \partial_{k} w_{i} u_{j} +  \partial_{i j} \bar{u}_{k}  w_{k}\partial_{j} w_{i} \: dx\nonumber \\
& \leq & \alpha_{1} \int_{\Omega} |\nabla^{2} w| |\nabla w | |u| \:dx  + \alpha_{1} \int_{\Omega} |\nabla^{2} \bar{u}| |\nabla w | |w| \:dx\nonumber \\
& \leq & \alpha_{1}  \|\nabla^{2} w\|_{L^{2}} \|\nabla w \|_{L^{2}}  \|u\|_{L^{\infty}}
 + \alpha_{1} \|\nabla^{2} \bar{u}\|_{L^{2}} \|\nabla w \|_{L^{2}}  \|w\|_{L^{\infty}}.
\end{eqnarray}
 In view of all these estimates,  we have
\begin{eqnarray}\label{bfinal}
\frac{1}{2} \frac{d}{d t} \|w(t)\|_{L^{2}}^{2} & + & \frac{\alpha_{1}}{2} \frac{d}{d t}\|\nabla w(t)\|_{L^{2}}^{2}
+ c_{1} c_{k} \|\nabla w(t)\|_{L^{2}} \nonumber\\ &\leq &  F( u, \bar{u}, w)\|\nabla w(t)\|_{L^{2}}^{2},
\end{eqnarray}
where $F(u, \bar{u}, w)$ is a function incorporating
$$\|\nabla u (t)\|_{L^{2}}^{2},\|\nabla^{2}
\bar{u}(t)\|_{L^{2}}^{2},  \|\nabla \bar{u}(t)\|_{L^{2}}^{2}, \|\nabla^{2} \bar{u}(t)\|_{L^{2}}^{2}, \|\nabla
w (t)\|_{L^{2}}^{2},\|\nabla^{2} w(t)\|_{L^{2}}^{2},$$
 and belonging to $L^{\infty}(0,T)$. Therefore,  we obtain
\begin{eqnarray}
\frac{d}{d t} \|w(t)\|_{L^{2}}^{2} +  \alpha \frac{d}{d t}\|\nabla w(t)\|_{L^{2}}^{2}
\leq  c  F( u, \bar{u}, w) (\|w(t)\|_{L^{2}}^{2} + \alpha  \|\nabla  w\|_{L^{2}}^{2}).
\end{eqnarray}
Finally, by Gronwall's inequality we infer that
\begin{eqnarray}
 \|w(t)\|_{L^{2}}= \|\nabla w(t)\|_{L^{2}}=0\;\;\;\;\text{ for almost every } t \: \in \: (0,T).
\end{eqnarray}
Consequently $u = \bar{u} $, and the proof is complete.

\section*{Acknowledgement}
We are grateful to the anonymous referee for a careful reading of the manuscript and fruitful remarks and suggestions.

\end{document}